\documentclass{article}
\usepackage[margin=1.7in]{geometry}

\usepackage{fancyhdr}
\usepackage{tikz,subfigure}
\usepackage{wrapfig}
\usepackage{amssymb}
\usepackage{amsmath}
\usepackage{amstext}
\usepackage{amsthm}
\usepackage{mathtools}
\usepackage{latexsym}
\usepackage{verbatim}
\usepackage{bbding}
\usepackage{linguex}
\usepackage{synttree}
\usepackage[round]{natbib}
\usepackage{multicol}
\usepackage{proof}
\usepackage[colorlinks=true,citecolor=blue,linkcolor=blue,urlcolor=blue]{hyperref}
\usetikzlibrary{decorations.pathreplacing,calligraphy}
\usepackage{authblk}
\usepackage{orcidlink} 

\tikzstyle{index on}=[inner sep=2pt, white, circle, fill=black]
\tikzstyle{index off}=[inner sep=2pt, black, circle, draw]
\tikzstyle{index gray}=[inner sep=2pt, black, circle, fill=lightgray]
\tikzstyle{opaque}=[fill=gray,fill opacity=.1]

\renewcommand{\phi}{\varphi}

\renewcommand{\gg}{\gamma}
\newcommand{\ee}{\varepsilon}

\renewcommand{\aa}{\alpha}
\newcommand{\bb}{\beta}

\newcommand{\existsi}{\mathord{\exists\hspace{-.4em}\exists}}

\def\disji{\rotatebox[origin=c]{-90}{$\!{\geqslant}$}}

\newcommand\M{\ensuremath{\mathcal{M}}}

\renewcommand{\o}{\overline}

\newcommand{\inqbq}{\ensuremath{\textsf{InqBQ}}}
\newcommand{\inqb}{\ensuremath{\textsf{InqB}}}
\newcommand{\lori}{\,\disji\,}

\newcommand{\id}{\text{id}}

\theoremstyle{definition}
\newtheorem{theorem}{Theorem}[section]
\newtheorem{definition}[theorem]{Definition}
\newtheorem{proposition}[theorem]{Proposition}

\newtheorem{corollary}[theorem]{Corollary}
\newtheorem{example}[theorem]{Example}

\newtheorem{lemma}[theorem]{Lemma}

\newcommand{\dep}[2]{=\hspace{-0.1cm}(#1,#2)}

\newcommand{\inqBT}{\ensuremath{\mathsf{InqBT}}}
\newcommand{\inqbt}{\ensuremath{\mathsf{InqBT}}}

\title{Inquisitive first-order logic is neither compact nor recursively axiomatizable\thanks{The first author gratefully acknowledges funding from the European Research Council under the EU
Horizon Europe research and innovation programme (Grant Agreement No. 101116774). The second author was partially supported by a Research Council of Finland grant No. 375116. 

No use of AI was made in this work, either in obtaining the results or in writing the paper.}}

\author{Ivano Ciardelli\footnote{University of Padua. Email: ivano.ciardelli@unipd.it. ORCID: 0000-0001-6152-3401.} 
{ and }Juha Kontinen\footnote{University of Helsinki. Email: juha.kontinen@helsinki.fi. ORCID: 0000-0003-0115-5154.}}

\date{\today}

\begin{document}

\maketitle

\begin{abstract}
Inquisitive first-order logic is an extension of classical first-order logic with formulas regimenting first-order questions, such as ``whether all objects are $P$'', ``which objects are $P$'', and ``what is one object that is $P$''. Since it was first developed in 2009, two major meta-theoretical questions about this logic have remained open, in spite of significant efforts. The first concerns compactness: if a conclusion follows from a set of premises, does it always follow from some finite subset? The second concerns the computational status of validity: is the set of validities recursively enumerable, or equivalently, does the logic admit a recursive axiomatization? 

We settle both questions in the negative, showing that inquisitive first-order logic is neither compact nor recursively axiomatizable. Furthermore, we prove that it violates another signature property of first-order logic, namely, Craig interpolation.
We discuss the significance of our results, and show how to extend them to a closely related logic, viz., inquisitive team logic.
\end{abstract}

\section{Introduction}

Traditionally, logic is concerned with relations between \emph{statements}---declarative sentences such as ``Alice is a student'' or ``all students passed the test''. Inquisitive logic \citep{Ciardelli:18book,Ciardelli:23book} is an approach that aims to extend the scope of logic to cover not only statements, but also \emph{questions}---interrogative sentences such as ``whether Alice is a student'' and ``which students passed the test''. 
This extension is interesting for several reasons: it makes it possible to regard interesting logical relations such as \emph{answerhood} and \emph{dependency} as facets of the key logical notion of \emph{entailment}, and to deploy the tools traditionally developed for entailment to study these notions \citep{Ciardelli:16dependency,Ciardelli:23book}; it allows for a uniform account of the contribution of connectives and quantifiers across statements and questions \citep{Ciardelli:18book}; it enables us to perform inferences with formulas denoting information \emph{types}, rather than specific pieces of information \citep{Ciardelli:18qait,PuncocharNoguera:26}; it allows for the definition of new modal operators that capture interesting modal notions
\citep{CiardelliRoelofsen:15idel,Ciardelli:25,Ciardelli:26action}. 
Since its development, inquisitive logic has found applications in different domains, including natural language semantics \citep[see, among many others,][]{CoppockBrochhagen:13,Anderbois:14,RoelofsenFarkas:15,Szabolcsi:15,Ciardelli:18counterfactuals,Booth:21,RoelofsenDotlacil:22}, cognitive psychology \citep{KoralusMascarenhas:13,MascarenhasKoralus:15,MascarenhasPicat:19}, formal epistemology \citep{Cohen:21} and decision theory \citep{DeverSchiller:20}.

Syntactically, inquisitive logics are typically obtained by taking an existing logic of statements, and extending it with question-forming operators. Semantically, they are interpreted in standard possible-world models for intensional logics. In order to give a uniform account of statements and questions, however, the basic semantic notion of \emph{truth at a world} is replaced by a notion of \emph{support at an information state}, where an information state (also called simply a \emph{state}) is modeled as a set of possible worlds. For a statement, to be supported at a state is to be true at each world in the state; for a question, to be supported is to be \emph{settled} in the state, i.e., to have an answer.\footnote{\label{fn:1}The move from single worlds to sets of worlds mirrors a parallel move made in the framework of \emph{team semantics} \citep{Hodges:97}, where variable assignments are replaced by sets of assignments. Team semantics has been used to define a variety of logics of dependence and independence, such as dependence logic \citep{Vaananen:07}, inclusion logic \citep{Galliani:12}, and independence logic \citep{GradelVaananen:13}. The similarity of the two semantic frameworks accounts for deep and systematic connections between inquisitive logics and team-based logics \citep[for discussion see, among others,][]{Yang:14,YangVaananen:16,Ciardelli:16dependency,Ciardelli:23book,Ciardelli:20inqi}.}

The two most widely studied systems of inquisitive logic are \emph{inquisitive propositional logic}, \inqb, and \emph{inquisitive first-order logic}, \inqbq. The former extends classical propositional logic with a question-forming connective $\lori$, called \emph{inquisitive disjunction}. Thus, in this system one has, in addition to formulas regimenting statements like ``it's raining'' ($r$) and ``it's snowing'' ($s$), also formulas regimenting the questions ``whether or not it's raining'' ($r\lori\neg r$, abbreviated $?r$) and ``whether it's raining or snowing'' ($r\lori s$). In a similar vein, the first-order system \inqbq\ extends classical first-order logic with $\lori$ and with an inquisitive existential quantifier, $\existsi$. In this system one has, in addition to standard first-order formulas like $\forall xPx$ regimenting statements like ``all students passed the test'', also formulas like $?\forall xPx$,  $\forall x?Px$, and $\existsi xPx$ regimenting, respectively, the questions ``whether all students passed the test'', ``which students passed the test'', and ``who is one student who passed the test''. Entailments in this logic capture interesting logical relations between types of information: for instance, the fact that the formulas $\exists xPx$ and $\forall x?Px$ jointly entail $\existsi xPx$ captures the fact that, given the information that an object with a certain property exists and a specification of which objects have that property, one is then in a position to give an example of an object with the relevant~property.

Inquisitive propositional logic is by now well-understood: an axiomatization is known \citep{CiardelliRoelofsen:11jpl}, as well as several other proof systems \citep{Frittella:16,ChenMa:17,Ciardelli:18qait,Muller:23}; \inqb\ has been systematically related to intermediate logics \citep{Ciardelli:09thesis,CiardelliRoelofsen:11jpl}, and weakenings based on non-classical logics of statements have been studied extensively \citep{Puncochar:16generalization,Puncochar:19,Puncochar:20,Ciardelli:20inqi}; algebraic approaches have been developed \citep{Bezhanishvili:21,Puncochar:21}, and properties such as structural completeness and interpolation have been investigated \citep{Iemhoff:16,FergusonPuncochar:25}.

By contrast, ever since its first introduction in \cite{Ciardelli:09thesis}, inquisitive first-order logic has remained rather mysterious, even with regards to its fundamental meta-theoretic properties. In particular, in spite of many investigations, the following fundamental questions are long-standing open problems:

\begin{itemize}
\item \textbf{Open Question 1: Entailment-compactness}\\
Is \inqbq\ entailment-compact, in the sense that whenever a conclusion $\psi$ follows from a set of premises $\Phi$, it also follows from a finite subset $\Phi_0\subseteq\Phi$?
\item \textbf{Open Question 2: Recursive enumerability}\\
Is the set of validities of \inqbq\ recursively enumerable? Equivalently, does \inqbq\ admit a recursive axiomatization? 
\end{itemize}
We may summarize the situation by saying that, at present, we do not know whether \inqbq\ is first-order-logic-like or second-order-logic-like with respect to its fundamental meta-theoretic properties. 

Positive answers to the two questions above have been obtained for two important fragments of \inqbq: the \emph{clant} (classical antecedent) fragment, where antecedents of implications are restricted to classical first-order formulas \citep{Grilletti:21}; and the \emph{rex} (restricted existential) fragment, where the inquisitive existential can appear only within the antecedent of an implication \citep{CiardelliGrilletti:22}. For the whole system \inqbq, however, the above questions are still open.

In this paper, we provide an answer to these long-standing open questions. We show that, in general (i.e., without specific restrictions on the signature), \inqbq\ is \emph{not} entailment compact and \emph{not} recursively axiomatizable. In other words, \inqbq\ is second-order-like with respect to its meta-theoretic properties. 

Our proof strategy is simple. Building on ideas we explored in recent work \citep{KontinenCiardelli:26}, we identify a sentence $\eta$ that, relative to a certain class of ``full'' models, expresses the finiteness of the domain. While the class of full information models is not definable in \inqbq, its complement is: in other words, there is a sentence $\theta$ capturing non-fullness. We can then produce a violation of entailment-compactness as follows: consider the set $\Phi$ containing $\eta$ along with each first-order sentence $\chi_n$ expressing that there are at least $n$ objects for $n\in\mathbb{N}$; this set can only be satisfied in a model which is not full, and so it entails $\theta$; but each finite subset $\Phi_0\subseteq\Phi$ can be satisfied in a full model, and so it does not entail $\theta$.

Using the formulas $\eta$ and $\theta$, we can also settle Open Question 2 by reducing the problem of finite validity (i.e., validity over finite structures) for first-order logic to the problem of validity in \inqbq: a first-order sentence $\aa$ is finitely valid if and only if the sentence $\eta\to\aa\lori\theta$ is valid in \inqbq\ (more precisely, in the version of \inqbq\ with rigid equality, as we will explain). Since the set of first-order validities over finite models is not recursively enumerable, neither is the set of \inqbq-validities. 

Our paper thus settles the two main open problems in the area of inquisitive first-order logic, while also opening new, interesting ones; in addition, our proof strategy provides ideas that may well find broader applicability in the study of this logic. 

The rest of the paper is structured as follows. In Section \ref{sec:background} we cover the necessary background on \inqbq. In Section \ref{sec:compactness} we show that \inqbq\ is not entailment-compact, settling Open Question 1. In Section \ref{sec:re} we show that \inqbq-validities are not recursively enumerable, settling Open Question 2. 
In  Section \ref{sec:relsig}
we generalize our main results to relational signatures without equality. In Section \ref{sec:craig} we exploit our findings to show the failure of the Craig interpolation for $\inqbq$.
In Section \ref{sec:inqbt} we extend our results to \emph{inquisitive team logic} ($\inqBT$), a system closely related to \inqbq\ but interpreted in the setting of team semantics (see Footnote \ref{fn:1}). 
Section \ref{sec:conclusion} summarizes our results and discusses further research directions.

\section{Background}
\label{sec:background}

This section provides the necessary background on \inqbq. For a more detailed introduction to this logic, and for proofs of the results, we refer to \cite{Ciardelli:23book}.

\paragraph{Syntax.} Just as in standard predicate logic, a signature $\Sigma$ consists of a set of predicate symbols and a set of function symbols, each with a corresponding arity $n\in\mathbb{N}$. 
Equality is viewed as a special binary predicate $=$, which may or may not be in the signature. A \emph{relational} signature is one that contains only predicate symbols.
Function symbols of arity 0 are called \emph{constant symbols}. 
The set of \emph{terms} is defined inductively from first-order variables and function symbols in the usual way. 
We denote inquisitive first-order logic over the signature $\Sigma$ by $\inqbq_\Sigma$, omitting reference to $\Sigma$ when convenient.

Formulas of $\inqbq_\Sigma$ are given by the following BNF definition, where $P$ is an $n$-ary predicate symbol and $t_1,\dots,t_n$ as well as  $t,t'$ are terms:
$$\phi\;:=\;P(t_1,\dots,t_n)\mid \bot\mid(\phi\land\phi)\mid(\phi\lori\phi)\mid(\phi\to\phi)\mid\forall x\phi\mid\existsi x\phi$$
Formulas that do not contain $\lori$ or $\existsi$ are called \emph{classical formulas}. The set of classical formulas can be identified with the language of standard first-order predicate logic, with a particular choice of primitives. The remaining operators can be defined in the following standard way:
\begin{multicols}{2}
\begin{itemize}
\item $\neg\phi:=\phi\to\bot$
\item $\top:=\neg\bot$
\item $(\phi\lor\psi):=\neg(\neg\phi\land\neg\psi)$
\item $\exists x\phi:=\neg\forall x\neg\phi$
\end{itemize}
\end{multicols}
\noindent
As usual, we write $(t=t')$ for an atomic formula involving the equality predicate $=$ applied to terms $t$ and $t'$, and we write $(t\neq t')$ for the negation $\neg(t=t')$. 

Following standard practice in inquisitive logic, a unary question-mark operator is defined as follows: 
$$?\phi:=(\phi\lori\neg\phi)$$
We use $\phi,\psi,\chi$ as meta-variables ranging over all formulas of \inqbq, and $\aa,\bb,\gamma$ as meta-variables ranging over classical formulas.

\paragraph{Models.} A \emph{first-order information model} for a signature $\Sigma$ is a triple $M=(W,D,I)$ where:
\begin{itemize}
\item $W$, the \emph{universe}, is a non-empty set whose elements are called \emph{worlds};
\item $D$, the \emph{domain}, is a non-empty set whose elements are called \emph{individuals};
\item $I$ equips each world $w$ with an interpretation function $I_w$, which assigns
\begin{itemize}
\item to each $n$-ary predicate symbol $P\in\Sigma$, an $n$-ary relation $I_w(P)\subseteq D^n$;
\item to each $n$-ary function symbol $f\in\Sigma$, an $n$-ary function $I_w(f): D^n\to D$.
\end{itemize}
For convenience, we also write $P_w$ or $f_w$ instead of $I_w(P)$ and $I_w(f)$.
\end{itemize}
If the equality predicate $=$ is part of the signature, its interpretation $=_w$ relative to each world is required to be an equivalence relation and a congruence with respect to all the predicate and function symbols in the signature, 
that is: 
\begin{itemize}
\item for any $n$-ary predicate $R\in\Sigma$ and $n$-tuples $(d_1,\dots,d_n)$ and $(d_1',\dots,d_n')$ with $d_1=_w d_1',\; \dots\;, d_n=_w d_n'$:
$$(d_1,\dots,d_n)\in R_w\iff (d_1',\dots,d_n')\in R_w;$$
\item for any $n$-ary function symbol $f\in\Sigma$ and $n$-tuples $(d_1,\dots,d_n)$ and $(d_1',\dots,d_n')$ with $d_1=_w d_1',\; \dots\;, d_n=_w d_n'$:  
$$f(d_1,\dots,d_n)=_w f(d_1',\dots,d_n').$$
\end{itemize}
Note that to each world $w$ in such a model $M$ we can associate a standard relational structure $\M_w=(D/{=_w},I_w^=)$ for the signature $\Sigma$, whose domain is the quotient of $D$ relative to the local equivalence relation $=_w$, and those interpretation function $I_w^=$ is the one induced on the quotient by $I_w$. 

Models where the equivalence relation $=_w$ is simply the identity relation $id_{D}=\{(d,d)\mid d\in D\}$ for each world are called id-models. In an id-model, the relational structure $\M_w$ associated with a world is simply the pair $(D,I_w)$; an id-model can thus be seen simply as a family $(\M_w)_{w\in W}$ of standard relational structures over a common domain.

\paragraph{Semantics.} An \emph{information state} (or simply a \emph{state}, for short) in a model $M=(W,D,I)$ is a set of worlds $s\subseteq W$. As usual, an assignment is a function from the set of first-order variables into $D$. If $g$ is an assignment, $x$ is a variable, and $d\in D$, we denote by $g[x\mapsto d]$ the assignment that maps $x$ to $d$ and coincides with $g$ on the remaining variables. The denotation of a term $t$ relative to a world $w$ and an assignment $g$, denoted $[t]_w^g$, is defined inductively by 
$$[x]_w^g=g(x),\qquad\qquad[f(t_1,\dots,t_n)]_w^g=f_w([t_1]_w^g,\dots,[t_1]_w^g).$$
As discussed in the introduction, the semantics of \inqbq\ is given in terms of a relation of \emph{support} which is assessed relative to information states. Given a first-order information model $M=(W,D,I)$, a state $s\subseteq W$, and an assignment $g$ from the set of variables into $D$, we define support by the following clauses:
\begin{itemize}
\item $M,s\models_g P(t_1,\dots,t_n)\iff \forall w\in s:([t_1]_w^g,\dots,[t_n]_w^g)\in P_w$
\item $M,s\models_g\bot\iff s=\emptyset$
\item $M,s\models_g\phi\land\psi\iff M,s\models_g\phi$ and $M,s\models_g\psi$
\item $M,s\models_g\phi\lori\psi\iff M,s\models_g\phi$ or $M,s\models_g\psi$
\item $M,s\models_g\phi\to\psi\iff \forall t\subseteq s: M,t\models_g\phi$ implies $M,t\models_g\psi$
\item $M,s\models_g\forall x\phi\iff\forall d\in D: M,s\models_{g[x\mapsto d]}\phi$
\item $M,s\models_g\existsi x\phi\iff\exists d\in D: M,s\models_{g[x\mapsto d]}\phi$
\end{itemize}
The support relation is persistent, i.e., if $M,s\models_g\phi$ and $t\subseteq s$, then $M,t\models_g\phi$. Moreover, the empty state trivially support any formula. As usual, the choice of the assignment $g$ does not matter in the case of sentences and closed terms, and we can thus omit reference to it.

If we say that a formula $\phi$ is supported relative to a model $M$, we mean that it is supported relative to the entire universe $W$ of the model. In symbols, we write $M\models_g\phi$ for $M,W\models_g\phi$. The following proposition guarantees that support at a state depends only on the worlds in the state, not on the remaining part of the model.

\begin{proposition}[Locality]\label{prop:locality} Given a model $M=(W,D,I)$ and a state $s\subseteq W$, define the restriction of $M$ to $s$ to be $M_{|s}:=(s,D,I_{|s})$, where $I_{|s}$ is the restriction of $I$ to $s$. Then for every assignment $g$ and formula $\phi$ we have:
$$M,s\models_g\phi\iff M_{|s},s\models_g\phi\iff M_{|s}\models_g\phi$$
\end{proposition}

\paragraph{Entailment, validity, and equivalence.} The logical notions of entailment, validity, and equivalence are defined in the obvious way: 
\begin{itemize}
\item A set of formulas $\Phi$ \emph{entails} a formula $\psi$, denoted $\Phi\models\psi$, if for any model $M$, state $s$ and assignment $g$, if $M,s\models_g\phi$ for all $\phi\in\Phi$, then $M,s\models_g\psi$.
\item A formula $\psi$ is \emph{valid}, denoted $\models\psi$, if it is entailed by the empty set, i.e., if it is supported relative to every model, state, and assignment. 
\item Two formulas $\phi$ and $\psi$ are \emph{equivalent}, denoted $\phi\equiv\psi$, if for every model $M$, state $s$, and assignment $g$ we have $M,s\models_g\phi\iff M,s\models_g\psi$. 
\end{itemize}
Corresponding notions of id-entailment ($\Phi\models_{\id}\psi$), id-validity ($\Phi\models_{\id}\psi$), and id-equivalence ($\phi\equiv_\id\psi$) are defined by restricting the class of models to id-models, i.e., models where `=' is interpreted at each world as the identity relation on the domain.

\paragraph{Truth and conservativity.} Although the primitive semantic notion for \inqbq\ is support at a state, a notion of truth at a world is recovered by defining truth at a world $w$ as support with respect to the corresponding singleton $\{w\}$: 
$$M,w\models_g\phi\overset{\text{def}}{\iff} M,\{w\}\models_g\phi$$
The following proposition ensures that, for classical formulas, the resulting truth-conditions coincide with the familiar ones.

\begin{proposition}\label{prop:conservative tc} Let $M$ be a model for \inqbq, $w$ a world, and $g$ an assignment. For any classical formula $\aa$ we have 
$$M,w\models_g\aa\iff \M_w\models_g\aa$$
where the expression on the right denotes satisfaction in standard Tarskian semantics.
\end{proposition}

For some formulas $\phi$, support simply boils down to global truth, in the sense that for every model $M$, state $s$, and assignment $g$ we have
$$M,s\models_g\phi\iff\forall w\in s: M,w\models_g\phi$$
If this is the case, we say that the formula $\phi$ is \emph{truth-conditional}. The following proposition ensures that, up to equivalence, the truth-conditional formulas of \inqbq\ are exactly the classical formulas.

\begin{proposition}\label{prop:tc} A formula of \inqbq\ is truth-conditional if and only if it is equivalent to a classical formula. In particular, classical formulas are truth-conditional.
\end{proposition}

\noindent
Together, propositions \ref{prop:conservative tc} and \ref{prop:tc} imply that with respect to classical formulas, our semantics is essentially equivalent to standard Tarskian semantics. As a consequence, for classical formulas entailment coincides with entailment is classical first-order logic.

\begin{proposition}[Conservativity] If $\Gamma$ is a set of classical formulas and $\aa$ is a classical formula, we have $$\Gamma\models\aa\iff \Gamma\models_{\textsf{FOL}}\aa$$where $\models_{\textsf{FOL}}$ denotes entailment in classical first-order logic.
\end{proposition}

\paragraph{Questions.} In contrast to classical formulas, formulas involving the inquisitive operators $\lori$ and $\existsi$ are typically not truth-conditional. For them, support at a state does not reduce to a local condition on the individual worlds in the state, but is instead a global matter. In this section, we introduce some particular examples of such formulas that play a key role in the following. 

\begin{example}[Mention-all questions] Let $P$ be an $n$-ary predicate and let $\overline x$ be a sequence of $n$ variables. The sentence $\forall\overline x?P(\overline x)$ is supported in a state iff $P$ has the same extension at each world in the state:
$$M,s\models_g\forall\overline x?P(\overline x)\iff \forall w,w'\in s: P_w=P_{w'}$$
Intuitively, $\forall\overline x?P(\overline x)$ captures the question ``which objects (or tuples) satisfy $P$'', a \emph{mention-all} question that asks for a complete specification of the extension of $P$. 
\end{example}

In particular, the formula $\forall x\forall y?(x=y)$ expresses  the rigidity of equality, i.e., fact that the equality predicate `=' is interpreted in the same way at each world in the evaluation state. Obviously, this formula is valid relative to id-models, where `=' is interpreted rigidly as the identity relation; in fact, it fully characterizes the logic of id-models relative to the logic of all first-order information models: a claim of id-entailment is equivalent to a claim of general entailment with the extra premise that equality is rigid \citep[for a proof, see Proposition 5.5.31 in][]{Ciardelli:23book}.

\begin{proposition}[General entailment and id-entailment]\label{prop:id-entailment} For any set of formulas $\Phi$ and any formula $\psi$:
$$\Phi\models_\id\psi\iff \Phi,\forall x\forall y?(x=y)\models\psi$$
\end{proposition}

\noindent
This connection will play an important role below: we will prove our key results in the simpler setting of id-models, and then extend them to the general case by means of this connection.

\begin{example}[Mention-some questions] Let $P$ be an $n$-ary predicate and let $\overline x$ be a sequence of $n$ variables. The sentence $\existsi\overline x P(\overline x)$ is supported in a state iff some particular tuple $\overline d$ satisfies $P$ in every world in the state:
$$M,s\models_g\existsi \overline xP(\overline x)\iff \exists \overline d\in D^n\forall w\in s:\o d\in P_w$$
Intuitively, $\existsi xP(\overline x)$ captures the question ``what is an example of an object (or tuple) satisfying $P$'', a so-called \emph{mention-some} question that asks for an instance of $P$. 
\end{example}

\begin{example}[Identification questions] For a term $t$, define corresponding formulas $\lambda t$ and $\mu t$ as follows:
$$\lambda t\;:=\;\existsi x(x=t)\qquad\qquad \mu t\;:=\;\forall x?(x=t) $$
where $x$ is an arbitrarily chosen variable not occurring in $t$. 
Relative to id-models, these formulas are equivalent: they are both supported in a state if the interpretation of $t$ is constant throughout the state. More formally, if $M$ is an id-model, $s$ a state, and $g$ an assignment, we have:
$$M,s\models_g \lambda t\iff M,s\models_g\mu t\iff \forall w,w'\in s: [t]_w^g=[t]_{w'}^g$$
Intuitively, $\lambda t$ and $\mu t$ are two different ways of capturing the identification question ``what is $t$'', which asks to pin down the denotation of the term. 
\end{example}

\paragraph{Dependencies.} \inqbq\ provides us with a systematic way of expressing dependencies as implications among questions: intuitively, if $\phi$ and $\phi'$ are formulas representing the questions $Q$ and $Q'$ respectively, then the implication $\phi\to\phi'$ is supported in a state in case, relative to that state, any answer to $Q$ yields a corresponding answer to $Q'$. We introduce a special case of dependence formulas that plays a key role below.

If $\o t=t_1,\dots,t_n$ is a sequence of terms and $t'$ is a term, we use the following abbreviation, adopted from Dependence Logic:
$$\dep{\o t}{t'}\quad:=\quad \lambda t_1\land \dots\land\lambda t_n\to \lambda t'$$
In the context of an id-model $M$, the above formula is supported in a state $s$ if, throughout $s$, the value of $t'$ is functionally determined by the values of $t_1,\dots,t_n$. More formally, we have
\begin{eqnarray*}
M,s\models_g \,\dep{\o t}{t'}&\iff & \forall w,w'\in s: [\,\o t\,]_w^g=[\,\o t\,]_{w'}^g\text{ implies }[t']_w^g=[t']_{w'}^g
\end{eqnarray*}
where $[\,\o t\,]_w^g=([t_1]_w^g\dots,[t_n]_w^g)$.
Note that, since the formulas $\lambda t$ and $\mu t$ are equivalent relative to id-models, the same result could also have been obtained by defining $\dep{\o t}{t'}$ as $\mu t_1\land\dots\land\mu t_n\to \mu t'$.\footnote{This is relevant since the definition of $\dep{\o t}{t'}$ by means of $\mu$, unlike the one that uses $\lambda$, yields a formula belonging to a known well-behaved fragment of \inqbq, namely, the rex fragment \citep{CiardelliGrilletti:22}.}

%
%

\section{\inqbq\ is not entailment-compact}
\label{sec:compactness}

In this section, we settle Open Question 1, proving that \inqbq\ is \emph{not}, in general, entailment-compact. 


Consider the signature $\Sigma=\{=_2,a_0,b_0\}$ containing only the equality predicate and two constant symbols $a$ and~$b$, and consider an id-model $M$ for this signature with set of worlds $W$ and domain~$D$. Relative to each world $w$, the interpretation function assigns to each constant an extension, $a_w\in D$ and $b_w\in D$. Now to an information state $s\subseteq W$ we can associate a binary relation $R_s\subseteq D^2$ defined by 
$$R_s=\{(a_w,b_w)\mid w\in s\}.$$
We say that the model $M=(W,D,I)$ is \emph{full} in case $R_W=D^2$. 

A fact that will be useful below is that to each domain $D\neq\emptyset$ we can associate a corresponding canonical full id-model $M_D$, defined in the following way.

\begin{definition}\label{def:canonical id-model} Given  $D\neq\emptyset$, the canonical full id-model for the signature $\Sigma$ based on $D$ is the id-model $M_D$ whose universe is $W_D=D^2$, whose domain is $D$, and whose interpretation function is such that for a world $w=(d,e)$ we have $a_w=d$ and $b_w=e$.
\end{definition}

\noindent
In this model, for each state $s\subseteq W_D=D^2$ we simply have $R_s=s$. 
In particular, $R_{W_D}=W_D=D^2$, showing that $M_D$ is a full model.

We first prove that, relative to the class of full id-models, $\inqbq_\Sigma$ can define the finiteness of the domain. More precisely, we  prove that relative to this class of models, the finiteness of the domain is captured by the following sentence of $\inqbq_\Sigma$:
$$\eta\;:=\;( \dep{a}{b}\, \wedge \dep{b}{a}\land \existsi x (x\neq b)\; \rightarrow\;  \existsi x  (x\neq a)).$$

\begin{lemma}\label{lemma:1}If $M$ is a full id-model with domain $D$: $M\models\eta\iff D$ is finite.\footnote{Recall that $M\models\eta$ abbreviates $M,W\models\eta$, where $W$ is the universe of $M$.}
\end{lemma}

\begin{proof} Suppose $M$ is a full id-model with domain $D$. We prove that $M\not\models\eta\iff D$ is infinite. By definition, $M\not\models\eta$ iff there is an $s\subseteq W$ that supports the antecedent of $\eta$ but not the consequent. As we discussed above, this state is associated with a relation $R_s=\{(a_w,b_w)\mid w\in s\}$. Spelling out the semantics, we find that:
\begin{itemize}
\item $M,s\models{\dep{a}{b}}\iff R_s$ is a function;
\item $M,s\models{\dep{b}{a}}\iff R_s$ is injective;
\item $M,s\not\models\existsi x(x\neq a)\iff \text{dom}(R_s)=D$;
\item $M,s\models\existsi x(x\neq b)\iff \text{ran}(R_s)\neq D$.
\end{itemize}
\noindent
So, a state $s$ falsifies the conditional $\eta$ just in case the associated relation $R_s$ is an injective function from $D$ to $D$ which is not surjective. Since $M$ is a full model, every binary relation $R\subseteq D\times D$ is represented as $R_s$ for some state $s\subseteq W$. So, in a full model, a state falsifying the conditional exists iff there exists an injective function defined on $D$ which is not surjective, which holds iff $D$ is infinite.
\end{proof}

\noindent
The fullness of an id-model cannot be defined in \inqbq, due to persistency: support for \inqbq-sentences is preserved when we move from a model $M$ to a sub-model $M'$ obtained by restricting the set of worlds to a subset $W'\subseteq W$. However, the property of fullness is \emph{not} preserved by this operation: if $M$ is full, the submodel $M'$ is not necessarily full. Rather, the property of fullness is preserved in the opposite direction, namely, upwards: if the sub-model $M'$ is full, the entire model $M$ must be full. But this means that the complementary property of \emph{non-fullness} is downward-persistent. We might thus expect it to be definable in \inqbq. And indeed, it is.

\begin{lemma}\label{lemma:2} For any id-model $M$, $M\text{ is full}\iff M\not\models\theta$, where:
$$\theta\;:=\;\existsi x\existsi y\neg(x=a\land y=b)$$
\end{lemma}

\begin{proof} We have the following:
\begin{eqnarray*}
M\models\theta&\iff &\exists d,e\in D\forall w\in W:(d,e)\neq (a_w,b_w)\\
&\iff &\exists d,e\in D:(d,e)\not\in R_W\\
&\iff &R_W\neq D^2\\
&\iff &M\text{ is not full } \qedhere
\end{eqnarray*}
\end{proof}

\noindent
Using the formulas $\eta$ and $\theta$, we can show the failure of compactness for id-entailment.

\begin{theorem}\label{theor:compactness} \inqbq\ over the signature $\Sigma=\{=_2,a_0,b_0\}$ is not entailment-compact relative to id-models. That is, there is a set of formulas $\Phi\cup\{\psi\}$ such that $\Phi\models_\id\psi$, but for all finite $\Phi_0\subseteq\Phi$, $\Phi_0\not\models_\id\psi$.
\end{theorem}

\begin{proof} Let the sentences $\eta$ and $\theta$ be defined as above. Consider for each $n$ a classical sentence $\chi_n$ saying that $D$ contains at least $n$ elements, and let $X=\{\chi_n\mid n\in\mathbb N\}$. We claim that we have the following counterexample to the compactness of id-entailment:
\begin{enumerate}
\item $X,\eta\models_\id\theta$
\item $X_0,\eta\not\models_\id\theta$ for all finite $X_0\subseteq X$
\end{enumerate}
To prove item 1, need to show that for all id-models $M$,
$$(M\models\chi_n\text{ for all $n\in\mathbb N$) and }M\models\eta \;\;\Longrightarrow\;\; M\models\theta$$
which is equivalent to:
$$(M\models\chi_n\text{ for all $n\in\mathbb N$) and }M\not\models\theta \;\;\Longrightarrow\;\; M\not\models\eta$$
So, suppose $M\models\chi_n$ for all $n\in\mathbb N$ and $M\not\models\theta$. By propositions \ref{prop:conservative tc} and \ref{prop:tc} we have $M\models\chi_n\iff \#D\ge n$. Since $M\models\chi_n$ for all $n\in\mathbb N$, it follows that $D$ is infinite. Since $M\not\models\theta$, Lemma \ref{lemma:2} ensures that $M$ is full. So, $M$ is a full id-model with an infinite domain. By Lemma \ref{lemma:1}, it follows that $M\not\models\eta$, as required.

To prove item 2, consider any finite $X_0\subseteq X$. Fix an arbitrary $k\ge 1$ such that $X_0\cap \{\chi_k,\chi_{k+1},\dots\}=\emptyset$ (since $X_0$ is finite, such a $k$ exists). Let $D$ be a set with $k$ elements and let $M_D$ be the canonical full id-model based on $D$, as given by Definition \ref{def:canonical id-model}. We have:
\begin{itemize}
\item $M_D\models\chi_n$ for all $\chi_n\in X_0$, since for all $\chi_n\in X_0$ we have $n<k=\#D$;
\item $M_D\models\eta$ by Lemma \ref{lemma:1}, since $M_D$ is full and $D$ is finite;
\item $M_D\not\models\theta$ by Lemma \ref{lemma:2}, since $M_D$ is full.
\end{itemize}
So, $M_D$ provides a countermodel to the entailment $X_0,\eta\models\theta$.
\end{proof}

\noindent
Using the connection between id-entailment and general entailment in \inqbq\ (Proposition \ref{prop:id-entailment}), we can also extend our result to general entailment.

\begin{theorem}\label{thm:1} \inqbq\ over the signature $\Sigma=\{=_2,a_0,b_0\}$ is not entailment-compact. That is, there is a set of formulas $\Phi\cup\{\psi\}$ such that $\Phi\models\psi$, but $\Phi_0\not\models\psi$ for all finite $\Phi_0\subseteq\Phi$.
\end{theorem}

\begin{proof} Suppose for a contradiction that entailment in \inqbq\ over the signature $\Sigma$ were compact. Then id-entailment would be compact as well, since using Proposition \ref{prop:id-entailment} and the supposed compactness of \inqbq-entailment we would have:
\begin{eqnarray*}
\Phi\models_\id\psi&\iff & \Phi,\forall x\forall y?(x=y)\models\psi\\
&\iff & \text{for some finite }\Phi_0\subseteq\Phi: \Phi_0,\forall x\forall y?(x=y)\models\psi\\
&\iff & \text{for some finite }\Phi_0\subseteq\Phi: \Phi_0\models_\id\psi
\end{eqnarray*}
But this would contradict the previous theorem.
\end{proof}
\noindent
We thus established an answer to Open Question 1: inquisitive first-order logic is not entailment-compact---at least, not over arbitrary signatures. 

This result precludes the possibility of a strongly complete proof system for \inqbq\ (if such a system existed, every entailment would be witnessed by a proof, which would use only a finite set of premises, implying entailment-compactness). We will now see that a weakly complete proof system for \inqbq\ is not possible, either.



%


\section{Recursive enumerability}
\label{sec:re}

In this section, we settle Open Question 2, proving that the set of \inqbq\ validities is \emph{not}, in general, recursively enumerable.

Towards this result, we generalize the notion of canonical full id-model given in Definition \ref{def:canonical id-model} in the following way.

\begin{definition}\label{def:cm generalized}
Let $\M=(D,I)$ be a standard predicate logic model in a signature $\Sigma$, and let $a,b$ be constant symbols which are not in $\Sigma$. The \emph{canonical full id-model based on $\M$} is the id-model $M_\M$ over the extended signature $\Sigma\cup\{=_2,a_0,b_0\}$ defined as follows: the universe is $W_D=D^2$; the domain is $D$; and the interpretation function is defined so that for each world $w=(d,e)$, we have:
$$
a_w=d
\qquad
\qquad
b_w=e
\qquad
\qquad
\sigma_w=I(\sigma)\text{ for all $\sigma\in\Sigma$}
$$
%
In words, the interpretation of the fresh constant symbols $a$ and $b$ changes from world to world, covering all possible pairs of objects, while the interpretation of the symbols in $\Sigma$ is constant and coincides with their interpretation in the original model $\M$.
\end{definition}

\noindent
Clearly, the model $M_\M$ is full. Moreover, for each world $w$, consider the local structure $\M_w=(D,I^+_w)$ induced by $M_\M$ on $w$: the reduct of this structure to the signature $\Sigma$ is simply the original model $\M=(D,I)$. Using this fact, we can show that for classical sentences $\aa$ in the original signature $\Sigma$, support in $M_\M$ coincides with truth in $\M$ under standard Tarskian semantics.

\begin{lemma}\label{lemma:support-truth} For any classical sentence $\aa$ in the signature $\Sigma$ we have:
$$M_\M\models\aa\iff \M\models\aa$$
\end{lemma}

\begin{proof} We have:
\begin{align*}
&M_\M\models\aa&&\iff&&M_\M,W_D\models\aa & \text{(definition)}\\
&&&\iff&&\forall w\in W_D: M_\M,w\models\aa& \text{(Prop.\ \ref{prop:tc})}\\
&&&\iff&&\forall w\in W_D: \M_w\models\aa& \text{(Prop.\ \ref{prop:conservative tc})}\\
&&&\iff&&\M\models\aa&\text{($\M_w\upharpoonright\Sigma=\M$)}
\end{align*}
\end{proof}

\noindent
Let us say that a first-order sentence is \emph{finitely valid} in first-order logic if it is valid over finite structures. 
We can now establish a connection between finite validity in first-order logic and id-validity as well as general validity in \inqbq.

\begin{lemma}\label{lemma:3} For any first-order sentence $\aa$ in a signature $\Sigma$ not containing the constants $a$ and $b$, the following are equivalent:
\begin{enumerate}
\item $\aa$ is finitely valid in first-order logic
\item $\eta\to\aa\lori\theta$ is id-valid in \inqbq \hfill where $\eta,\theta$ are defined as Section \ref{sec:compactness}
\item $\rho\land\eta\to\aa\lori\theta$ is valid in \inqbq \hfill where $\rho:=\forall x\forall y?(x=y)$
\end{enumerate}
\end{lemma}

\begin{proof} $\,$
\begin{description}
\item[$1\Rightarrow 2$.] Suppose $\eta\to\aa\lori\theta$ is not id-valid in \inqbq. Then there is an id-model $M=(W,D,I)$ such that $M\models\eta$ but $M\not\models\aa$ and $M\not\models\theta$. Since $M\not\models\theta$, $M$ is full by Lemma \ref{lemma:2}. Since $M$ is full and $M\models\eta$, we have that $D$ is finite by Lemma \ref{lemma:1}. Since $M\not\models\aa$ and $\aa$ is classical, by Proposition \ref{prop:tc} there is a world $w$ such that $M,w\not\models\aa$. By Proposition \ref{prop:conservative tc}, this means that for the structure $\M_w=(D,I_w)$ we have $\M_w\not\models\aa$. So, $\M_w$ is a finite structure falsifying $\aa$, showing that $\aa$ is not finitely valid in first-order logic.

\item[$2\Rightarrow 1$.] Suppose $\aa$ is not finitely valid in first-order logic. Then there is a finite structure $\M$ with $\M\not\models\aa$. Now consider $M_\M$, the canonical full id-model based on $\M$, as given by Definition \ref{def:cm generalized}. By Lemma \ref{lemma:support-truth} we have $M_\M\not\models\aa$. By Lemma \ref{lemma:2}, we have $M_\M\not\models\theta$ since $M_\M$ is full, and by Lemma \ref{lemma:1} we have $M_\M\models\eta$ since $M_\M$ is full and finite. Thus, $M_\M\not\models\eta\to\aa\lori\theta$, showing that $\eta\to\aa\lori\theta$ is not id-valid in \inqbq.

\item[$2\Leftrightarrow 3$.] By Proposition \ref{prop:id-entailment} and the deduction theorem, for any formula $\phi$ we have $\models_\id\phi\iff  \rho\models\phi\iff {\models\rho\to\phi}$. The result then follows from this and the import-export equivalence $\rho\to(\eta\to\aa\lori\theta)\equiv\rho\land\eta\to\aa\lori\theta$.\qedhere
\end{description}
\end{proof}

\noindent
We can now use this connection to provide an answer to Open Question 2. 

\begin{theorem}\label{thm:2} 
The sets of validities and id-validities of \inqbq\ are not, in general, recursively enumerable. 
More precisely, if $\Sigma$ is a signature containing either a predicate symbol of arity $\ge 2$ (other than equality), or a function symbol of arity $\ge 1$, 
then the set of \inqbq-validities and id-validities in the extended signature $\Sigma\cup\{=_2,a_0,b_0\}$  is not recursively enumerable.
\end{theorem}

\begin{proof} The previous lemma shows that the problem of finite validity for first-order sentences in a signature $\Sigma$ reduces to the problem of id-validity in \inqbq\ over the extended signature $\Sigma\cup\{=_2,a_0,b_0\}$ through the computable mapping $\aa\mapsto{(\eta\to\aa\lori\theta)}$, and to the problem of general validity in this signature via the computable mapping $\aa\mapsto {(\rho\land\eta\to\aa\lori\theta)}$. 

Thus, if the set of first-order finite validities in a signature $\Sigma$ is not r.e., then the set of \inqbq\ validities and id-validities in the signature $\Sigma\cup\{=_2,a_0,b_0\}$ is not r.e., either. The claim then follows from the well-known fact that the set of first-order finite validities is not r.e.\ for signatures containing a predicate of arity $\ge 2$, or containing a function symbol of arity $\ge 1$ along with equality \citep{trakhtenbrot}. 
\end{proof}

\section{Extension to pure relational signatures}\label{sec:relsig}

The formulas $\eta$ and $\theta$ that we used in the previous sections to prove our results make use of two constants $a$ and $b$ in combination with the equality predicate. In this section we show that neither aspect of the signature is crucial: analogous results can be obtained for signatures containing only predicate symbols, and without equality.

\paragraph{Doing without constants.} 
To adapt our results to a relational signature (i.e., one containing only predicate symbols), it suffices to simulate the constants $a$ and $b$ by unary predicates $P$ and $Q$ that are stipulated to hold of exactly one object at each world.

Let us say that a model $M$ for a signature including the predicates $P$ and $Q$ is a \emph{unique-instance} model if for every world $w\in W$, $P_w$ and $Q_w$ are singletons. Note that $M$ is a unique-instance model just in case it supports the following classical first-order sentence: 
\begin{eqnarray*}
\iota&:=& \exists x\forall y(Py\leftrightarrow y=x)\land\exists x\forall y(Qy\leftrightarrow y=x)
\end{eqnarray*}
In a unique-instance model, a state $s\subseteq W$ can be associated, as above, with a binary relation on $D$, viz., $R_s=\{(a,b)\in D^2\mid \text{for some }w\in s: P_w=\{a\}\text{ and }Q_w=\{b\}\}$. Just as above, we then say that $M$ is \emph{full} in case $R_W=D^2$.

Relative to unique-instance models, the property of non-fullness is defined by the following adaptation of the formula $\theta$ above:
\begin{eqnarray*}
\theta'&:=&\existsi x\existsi y\neg(P(x)\land Q(y))
\end{eqnarray*}
Finally, we can adapt the definition of $\eta$ in the following way:
\begin{eqnarray*}
\eta'&:=&(\,{\dep{P}{Q}}\,\land\, {\dep{Q}{P}}\,\land\, {\existsi x\neg Q(x)}\;\to\; x\neg P(x)\,)
\end{eqnarray*}
where $\dep{P}{Q}:=(\forall x?P(x)\to\forall x?Q(x))$ and $\dep{Q}{P}:= (\forall x?Q(x)\to\forall x?P(x))$.
A straightforward adaptation of the proof of Lemma \ref{lemma:1} establishes that for every full unique-instance model $M=(W,D,I)$ we have
$M\models\eta'\iff D\text{ is finite}$.


%

With this result at hand, we can reproduce our main results in the relational signature $\Sigma=\{=_2,P_1,Q_1\}$.

\begin{corollary} \label{cor:compactness-relational} \inqbq\ over the signature $\Sigma=\{=_2,P_1,Q_1\}$ is not entailment-compact. 
\end{corollary}

\begin{proof} Define $X$ as in the proof of Theorem \ref{theor:compactness}. A straightforward adaptation of that proof  shows that $X,\iota,\eta'\models_\id\theta'$, but  $X_0,\iota,\eta'\not\models_\id\theta'$ for all finite $X_0\subseteq X$.
This shows that id-entailment in the signature $\Sigma$ is not compact. The extension from id-entailment to general entailment then follows via Proposition \ref{prop:id-entailment}.
\end{proof}

\begin{corollary}\label{cor:axiomatizability-relational} 
Let $\Sigma$ be a signature containing either a predicate of arity $\ge 2$ (other than equality), or a function symbol of arity $\ge 1$. 
Let $\Sigma^+=\Sigma\cup\{=_2,P_1,Q_1\}$ be the signature obtained by adding to $\Sigma$ two fresh unary predicates and, if not already present in $\Sigma$, the equality predicate. Then the set of validities and id-validities of \inqbq\ in the signature $\Sigma^+$ is not recursively enumerable.
\end{corollary}

\begin{proof} We reproduce the proof of Lemma \ref{lemma:3}, showing that for every classical formula $\aa$ in the signature $\Sigma$, the following are equivalent:
\begin{itemize}
\item $\aa$ is finitely valid;
\item $\iota\land\eta'\to\aa\lori\theta'$ is id-valid in \inqbq;
\item $\rho\land\iota\land\eta'\to\aa\lori\theta'$ is valid in \inqbq. \hfill where $\rho:=\forall x\forall y?(x=y)$
\end{itemize}
Note that the \inqbq-formulas in the second and third line are in the signature $\Sigma^+$.
\end{proof}

\paragraph{Doing without equality.} To avoid relying on the equality predicate `$=$', note that 
in \inqbq, this predicate is interpreted merely as a congruence. The logical behavior of this predicate can thus be simulated by an arbitrary binary predicate $E$ along with first-order axioms that force the interpretation of this predicate to be a congruence.

Given a binary predicate $E$, let $\ee_E$ be the classical first-order sentence expressing that $E$ is an equivalence relation. Moreover, for any symbol $\sigma$ in the signature, let $\kappa_E^\sigma$ be the first-order sentence expressing that $E$ is a congruence with respect to $\sigma$. Finally, let $K_E^\Sigma=\{\ee_E\}\cup\{\kappa_E^\sigma\mid\sigma\in\Sigma\}$. Note that any model $M$ supporting $K_E^\Sigma$ is one where at each world $w$, the interpretation $E_w$ is a congruence for all symbols in the signature, and can thus in principle serve as an interpretation for the equality predicate. Using this fact, it is straightforward to prove the following lemma.

%

\begin{lemma}\label{lemma:congruence} Let $\Phi\cup\{\psi\}$ be a set of formulas in a signature containing the equality predicate `$=$'. Let $\Phi_E$ and $\psi_E$ be obtained from $\Phi$ and $\psi$ by replacing each occurrence of the equality predicate by $E$. Then we have:
$$\Phi\models\psi\iff \Phi_E,K_E^\Sigma\models\psi_E.$$
\end{lemma}

\noindent
Using this connection, we can now easily prove versions of our results for signatures without equality.

\begin{corollary} Let $\Sigma=\{E_2,P_1,Q_1\}$ be a signature containing a binary predicate $E$ and two unary predicates $P,Q$. $\inqbq$ is not entailment-compact over this signature.
\end{corollary}

\begin{proof} 
If \inqbq\ were entailment-compact over the signature $\Sigma=\{E_2,P_1,Q_1\}$, it would be also over the signature $\Sigma'=\{=_2,P_1,Q_1\}$, in contrast to Corollary~\ref{cor:compactness-relational}. Indeed, given arbitrary $\Phi$ and $\psi$ in the signature $\Sigma'$, we would have:
\begin{align*}
&\Phi\models\psi&&\iff&&\Phi_E,K_E^\Sigma\models\psi_E & \text{(Lemma \ref{lemma:congruence})}\\
&&&\iff&&\text{for some finite }\Phi_0\subseteq\Phi:(\Phi_0)_E,K_E^\Sigma\models\psi_E& \text{(compactness over $\Sigma$)}\\
&&&\iff&&\text{for some finite }\Phi_0\subseteq\Phi:\Phi_0\models\psi& \text{(Lemma \ref{lemma:congruence})}
\end{align*}
%
\end{proof}

\begin{corollary} 
Let $\Sigma$ be a finite signature containing either a predicate of arity $\ge 2$ (other than equality), or a function symbol of arity $\ge 1$. 
Let $\Sigma^+=\Sigma\cup\{E_2,P_1,Q_1\}$ be the signature obtained by adding to $\Sigma$ a fresh binary predicate and two fresh unary predicates. Then the set of \inqbq-validities in the signature $\Sigma^+$ is not recursively enumerable.
\end{corollary}

\begin{proof} Since $\Sigma$ is a finite signature, $K_E^\Sigma$ is a finite set, so we can let $\kappa_E^\Sigma:=\bigwedge K_E^\Sigma$. 
By Lemma \ref{lemma:congruence}, for every formula $\phi$ in the signature $\Sigma\cup\{=_2,P_1,Q_1\}$ we have:
$$\models\phi\quad\iff\quad \kappa_E^\Sigma\models\phi_E\quad\iff\quad \models\kappa_E^\Sigma\to\phi_E$$
So, the problem of \inqbq-validity in the signature $\Sigma\cup\{=_2,P_1,Q_1\}$ reduces to the problem of \inqbq-validity in the signature $\Sigma\cup\{E_2,P_1,Q_1\}$ through the computable map $\phi\mapsto (\kappa_E^\Sigma\to\phi_E)$. Since the set of \inqbq-validities in the former signature is not r.e.\ by Corollary \ref{cor:axiomatizability-relational}, neither is the set of \inqbq-validities in the latter signature.
\end{proof}


\section{Failure of Craig interpolation}\label{sec:craig}


In addition to compactness and the recursive enumerability of validities, a signature property of classical first-order logic is the Craig interpolation theorem. This theorem states that, given sentences $\aa$ and $\beta$ in the signatures $\Sigma_\aa$ and $\Sigma_\bb$ respectively, if $\aa$ entails $\bb$ then there is an \emph{interpolant}, i.e., a formula $\gg$ in the signature $\Sigma_\aa\cap\Sigma_\bb$ such that $\aa$ entails $\gg$ and $\gg$ entails $\bb$. Our results allow us to show that this property, too, fails for \inqbq, if no language restrictions are imposed. Indeed, as the following theorem shows, interpolation fails in a rather dramatic way.


\begin{theorem} There exist formulas $\aa$ and $\phi$ of \inqbq\ in the signatures $\Sigma_\aa$ and $\Sigma_\phi$ respectively, such that:
\begin{enumerate}
\item $\aa$ entails $\phi$: $\aa\models\phi$;
\item the entailment is non-trivial: $\aa\not\models\bot$ and $\top\not\models\phi$;
\item $\aa$ and $\phi$ are in disjoint signatures: $\Sigma_\aa\cap\Sigma_\phi=\emptyset$. 
\end{enumerate}
Together, (1)-(3) imply a failure of Craig interpolation: by (3), an interpolant for $\aa\models\phi$ would have to be in the empty signature, and so equivalent to $\top$ or $\bot$;\footnote{Note that we take equality to be a predicate in the signature, so the fact that $\Sigma_\aa\cap\Sigma_\phi=\emptyset$ means that the only atomic formula in the common signature is $\bot$.} but by (2), neither $\top$ nor $\bot$ is an interpolant.
\end{theorem}

\begin{proof} Let $\aa$ be a sentence of classical first-order logic containing only a binary predicate $R$ (and no equality), such that $\aa$ is satisfiable but only admits infinite models.\footnote{We can, e.g., take the following sentence, which says that $R$ is irreflexive, transitive, and serial, and thus requires the existence of an infinite $R$-path:
$$\aa \;:= \;\forall x\neg R(x,x) \;\land\; \forall x\forall y\forall z(R(x,y) \land R(y,z) \to R(x,z)) \;\land\; \forall x\exists y R(x,y) $$}
Let $\phi\;:=\;\eta\land\rho\to\theta$, where $\eta$ and $\theta$ are defined as in Section~\ref{sec:compactness} and $\rho:=\forall x\forall y?(x=y)$. We claim that these formulas satisfy the conditions (1)-(3) above.
\begin{enumerate}
\item $\aa\models\phi$. By the deduction theorem, this claim is equivalent to $\aa,\eta,\rho\models\theta$, which in turn by Proposition \ref{prop:id-entailment} is equivalent to $\aa,\eta\models_\id\theta$. Towards a contradiction, suppose this is not the case: then there is an id-model $M=(W,D,I)$ and a state $s\subseteq W$ with $M,s\models\aa$ and $M,s\models\eta$ but $M,s\not\models\theta$.

Since $M,s\not\models\theta$ we know $s\neq\emptyset$. Fix a world $w\in s$.  By persistency, from $M,s\models\aa$ we get $M,w\models\aa$, and since $\aa$ is a classical formula, by Proposition~\ref{prop:conservative tc} we have $\M_w\models\aa$ in classical Tarskian semantics, where $\M_w=(D,I_w)$.\footnote{Note that the domain of the local structure $\M_w$ is simply $D$ since $M$ is an id-model.} By our choice of $\aa$, this means that $D$ is infinite. 

On the other hand, we have $M,s\models\eta$ and $M,s\not\models\theta$. By Proposition \ref{prop:locality}, these claims can be rewritten as global support claims ${M_{|s}}\models\eta$ and $M_{|s}\not\models\theta$ about the model $M_{|s}=(s,D,I_{|s})$ obtained by restricting $M$ to $s$. By Lemma \ref{lemma:2}, $M_{|s}\not\models\theta$ implies that $M_{|s}$ is full. By Lemma \ref{lemma:1}, $M_{|s}\models\eta$ together with the fullness of $M_{|s}$ implies that $D$ is finite, contradicting our earlier conclusion.

\item $\aa\not\models\bot$ since $\aa$ is satisfiable in Tarskian semantics, and so also satisfiable by some non-empty state of some model $M$ by Proposition \ref{prop:conservative tc}. 

$\top\not\models\phi$ since an arbitrary id-model $M$ that is full and based on a finite domain supports $\eta$ by Lemma \ref{lemma:1} but not $\theta$ by Lemma \ref{lemma:2}, thus falsifying $\phi$.
\item $\Sigma_\aa\cap\Sigma_\phi=\emptyset$ since $\Sigma_\aa=\{R\}$ and $\Sigma_\phi=\{=,a,b\}$.
\qedhere
\end{enumerate}
\end{proof}


\section{Transferring the results to inquisitive team logic}
\label{sec:inqbt}


Inquisitive first-order logic has a ``twin'' system, \emph{inquisitive team logic} ($\inqBT$), which is based on the same language, but interpreted in a different semantic setting: while the semantics of \inqbq\ is based on possible world semantics, representing multiple states of affairs, the one of \inqbt\ relies on team semantics (see Footnote \ref{fn:1}), representing multiple assignments of values to variables. More precisely, instead of information states---i.e., sets of worlds---the semantics of \inqbt\ uses \emph{teams}---i.e., sets of assignments. The semantic clauses are structurally parallel to those of \inqbq, but with teams now playing the role of information states. Just like in \inqbq\ we can formalize questions about the state of affairs, in \inqbt\ we can formalize questions about the values of variables; in particular, the question ``what is the value of $x$'' can be expressed by the formula $\lambda x\,:=\,\existsi y(y=x)$ (or, equivalently, by $\mu x\,:=\,\forall y?(y=x)$). Functional dependencies between the values of variables can be expressed as implications between questions: borrowing the notation of Dependence Logic \citep{Vaananen:07}, we can take $\dep{x_1,\dots,x_n}{y}$ to abbreviate the formula $\lambda x_1\land\dots\land\lambda x_n\to \lambda y$.\footnote{It is worth remarking that the system \inqbt\ was first considered by \cite{Yang:14} under the name \textsf{WID} (for \emph{weak intuitionistic dependence logic}). \cite{Yang:14} also observed that, with respect to sentences, \inqbt\ is essentially equivalent to classical first-order logic, modulo notational differences. Crucially, this does not extend to open formulas.}

The theoretical landscape about \inqbq\ and \inqbt\ is very similar. In particular, the two long-standing open problems that we have addressed for \inqbq\ correspond to analogous open problems for \inqbt: it is not known whether the logic \inqbt\ is entailment-compact, nor whether the set of its validities is recursively enumerable. 

Indeed, in recent work, \cite{CiardelliConti:26} have established entailment-preserving translations in both directions between \inqbq\ and \inqbt, and used this to show the equivalence of the key problems for the two systems: \inqbq\ is entailment-compact iff \inqbt\ is, and \inqbq-validities are r.e.\ iff \inqbt-validities are. Through this connection, our results for \inqbq\ transfer immediately to \inqbt. 
However, in this section we describe more explicitly how our proof strategy in this paper can be directly adapted to \inqbt, since this is of some independent interest (it allows us, e.g., to obtain the desired results for simpler signatures). We begin by providing the necessary background about \inqbt\ \citep[for a more detailed exposition, see \S7\ of][]{Ciardelli:23book}, and then sketch how to adapt the proofs of our main results.

\paragraph{Background on \inqbt.} The syntax of \inqbt\ is exactly the same as for \inqbq; for simplicity, throughout this section we assume that equality is in the signature.
%
%
%
%
%
%
%
%
%
%
The semantics of $\inqBT$ is given in terms of a relation of support, which is defined relative to a standard first-order structure $\M=(D,I)$ and a team $X$, which is 
 a set of assignments defined over a common domain $\text{dom}(X)$ of variables. 
For a team $X$, $x$ a variable, and $d\in D$ an object, we denote by $X[x\mapsto d]$ the team over the domain $\text{dom}(X)\cup\{x\}$ obtained by assigning to $x$ the constant value $d$ through the team:
$$X[x\mapsto d]=\{g[x\mapsto d]\mid g\in X\}.$$

\begin{definition} The relation $\M\models_X\phi$ between a model $\M$, a team $X$, and a formula $\phi$ whose free variables are included in $\text{dom}(X)$ is defined inductively as follows:
\begin{itemize}
    \item $\M\models_X p\iff \forall g\in X:\M\models_g p$ in Tarskian semantics\hfill if $p$ is an atom
    \item $\M\models_X\bot\iff X=\emptyset$
    \item $\M\models_X\phi\land\psi\iff\M\models_X\phi$ and $\M\models_X\psi$
    \item $\M\models_X\phi\lori\psi\iff\M\models_X\phi$ or $\M\models_X\psi$
    \item $\M\models_X\phi\to\psi\iff\forall Y\subseteq X:\M\models_Y\phi$ implies $\M\models_Y\psi$
    \item $\M\models_X\forall x\phi\iff\forall d\in D:\M\models_{X[x\mapsto d]}\phi$
    \item $\M\models_X\existsi x\phi\iff\exists d\in D:\M\models_{X[x\mapsto d]}\phi$
\end{itemize}
\end{definition}
\noindent
Just like for \inqbq, support for \inqbt\ is persistent (if a team $X$ supports a formula, so does every subteam $Y\subseteq X$) and the empty team supports any formula whatsoever. Moreover, support is \emph{local} in the sense that it depends only on the values that the team assigns to the variables that occur free in the relevant formula. More precisely, for any model $\M$, any formula $\phi$ of $\inqBT$, and any teams $X,Y$ whose domain includes the set $\text{FV}(\phi)$ of free variables in $\phi$, we have
$$X|_{\text{FV}(\phi)}=Y|_{\text{FV}(\phi)}\text{ implies }\M\models_X\phi\iff\M\models_Y\phi,$$
where $X|_{\text{FV}(\phi)}$ denotes the team obtained by restricting each assignment in $X$ to $\text{FV}(\phi)$, and likewise for $Y$.
%
By the locality property, we can stipulate that a model $\M$ satisfies a sentence $\phi$ (written $\M\models\phi$) in case $\M\models_X\phi$ holds relative to an arbitrary nonempty team $X$; if $\M\not\models\phi$ we say that $\M$ falsifies $\phi$. 

The notions of entailment, validity, and equivalence are defined in the obvious~way, e.g.,  a set of formulas $\Phi$ entails a formula $\psi$ ($\Phi\models\psi$), if in every model, any team that is defined on all the free variables in $\Phi\cup\{\psi\}$ and supports all $\phi\in\Phi$ also supports $\psi$.

The connection with classical first-order logic can be  made through a  semantic property of \emph{flatness} which is the analogue of truth-conditionality of $\inqbq$-formulas.
    A formula $\phi$ is flat if for all models $\M$ and teams $X$ with $\text{dom}(X)\supseteq\text{FV}(\phi)$ we have
    $$\M\models_X\phi\iff\forall g\in X:\M\models_{\{g\}}\phi.$$

\noindent One can then show that all classical formulas are flat, and moreover, $\M\models_{\{g\}}\phi$ holds in \inqbt\ just in case $\M\models_{g}\phi$ holds in standard Tarskian semantics. Thus, for a classical formula, to be supported by a team is just to be true, classically, relative to each assignment in the team. In particular, for a classical sentence, our definition of satisfaction in a model coincides with the one given by Tarskian semantics.



\paragraph{Extending the results to \inqbt.}
We will show how to redefine the crucial formulas $\eta$ and $\theta$, using which the results for \inqbt\ can be established by straightforward adaptations of the arguments we gave for \inqbq. Let:
$$\eta(x,y)\;:=\;( \dep{x}{y}\, \wedge \dep{y}{x}\land \existsi z (z\neq x)\; \rightarrow\;  \existsi z  (z\neq y)).$$
where, as discussed above in this section, $\dep{x}{y}$ is defined as a shorthand for the implication $\lambda x\to\lambda y$, and similarly for $\dep{y}{x}$. Note that with this definition, $\eta$ is a formula with two free variables $x,y$ in the signature $\Sigma=\{=_2\}$ containing only the equality predicate. Reasoning as in the proof of Lemma \ref{lemma:1} we have that
$\M\models_X\phi\iff D \textrm{ is finite}$,
where $X$ is the \emph{full} team with domain $\{x,y\}$ over $D$, i.e., the team containing all assignments from $\{x,y\}$ into $D$ \citep[see also Theorem 3.1 in][]{KontinenCiardelli:26}. Furthermore, consider the formula $\theta$ defined as
$$\theta\;:=\;\existsi u\existsi v\neg(x=u\,\land\, y=v),$$
Reasoning as in the proof of Lemma \ref{lemma:2}, we can show that for any team $X$ over $\{x,y\}$, $\M\models_X\phi$ holds iff $X$ is not full, i.e., if there is some assignment from $\{x,y\}$ into $D$ that is not in $X$.

With these modified formulas at hand, we can closely mimic the proofs of Theorem \ref{thm:1} and Theorem \ref{thm:2} to obtain the following results.

\begin{theorem} \inqBT\ over the signature $\Sigma:=\{=\}$ is not entailment-compact.
\end{theorem}

\begin{theorem} 
If $\Sigma$ is a signature containing, in addition to equality, either a predicate symbol of arity $\ge 2$ or a function symbol of arity $\ge 1$, 
then the set of \inqBT-validities in the signature  $\Sigma$  is not recursively enumerable.
\end{theorem}


\section{Conclusion and open problems}
\label{sec:conclusion}

Since inquisitive first-order logic, \inqbq, was first introduced in \cite{Ciardelli:09thesis}, two fundamental meta-theoretic questions about this logic have remained open: whether entailment is compact, and whether validities are recursively enumerable. We have settled both questions in the negative: \inqbq-entailment is \emph{not} compact, and \inqbq-validities are \emph{not} recursively enumerable. Furthermore, we have proved that another key property of classical first-order logic, namely, Craig interpolation, fails in \inqbq. The situation can be summarized by saying that, in terms of its basic meta-theoretic properties, \inqbq\ patters with second-order rather than first-order logic.

It is worth pausing for a moment to reflect on \emph{why} this should be the case. \inqbq\ is intended primarily to provide a logic of \emph{first-order} questions---questions obtained by quantifying over ordinary individuals, such as ``which objects are $P$'' and ``what is one object that is $P$''. These questions do not seem to involve any recourse to second-order quantification. Why, then, should \inqbq\ have second-order-logic-like features? The reason is that \inqbq\ does not only extend first-order logic with formulas regimenting ordinary first-order questions: it also allows us to combine such formulas freely by logical operators, notably by means of implication. As we mentioned above, an implication between questions expresses the existence of a \emph{dependency}---essentially, a \emph{function} mapping answers to the antecedent question to corresponding answers to the consequent question. Such a function is inherently a second-order object. It is not questions \emph{per se}, then, but implications between questions, that introduce a second-order component into the semantics. This diagnosis leads us to expect that, by keeping the question-forming operations of \inqbq\ unchanged, but disallowing the possibility of connecting questions by implications, we would obtain a first-order-like fragment. And indeed, this is so: as mentioned in the introduction, the \emph{clant} (classical antecedent) fragment of \inqbq, obtained by disallowing questions in the antecedent of an implication, is compact and recursively axiomatizable, as shown by \cite{Grilletti:19} \citep[see also \S5.7 and \S6.3\ in][]{Ciardelli:23book}. It is worth noting that this fragment includes those formulas of \inqbq\ that correspond to natural classes of first-order questions, such as $?\forall xP(x)$ (``whether every object is $P$''), $\forall x?P(x)$ (``which objects are $P$''), $\existsi xP(x)$ (``what is one object that is $P$''), and many more.\footnote{On the other hand, note that the \emph{rex} fragment of \inqbq, obtained by restricting the occurrences of $\existsi$ to conditional antecedents, is also known to be first-order-like \citep{CiardelliGrilletti:22}, while being closed under arbitrary implications. This shows that the possibility of expressing dependencies among first-order questions does not \emph{inevitably} lead to a genuinely second-order semantics: it depends crucially on the particular classes of first-order questions that a certain logic can express.}


While settling some long-standing questions, our results also raise some new ones. 

First, it would be interesting to investigate whether there are any (non-trivial) signature types to which our results do not apply. In particular, what about the \emph{monadic} fragment of \inqbq, where the signature contains only unary predicates? Our proof strategy does not seem viable without a binary predicate in the signature, which leaves open the possibility that the monadic fragment of \inqbq\ is entailment-compact and recursively axiomatizable.

Second, in light of our results, work on ``tame'' fragments of \inqbq---fragments that \emph{are} entailment-compact and whose validities are recursively enumerable---becomes especially pressing. As mentioned in the introduction, at present two main tame fragments are known: the clant (classical antecedent) fragment \citep{Grilletti:19}, which restricts antecedents of implications to classical formulas, and the rex (restricted existential) fragment \citep{CiardelliGrilletti:22}, which restricts occurrences of the inquisitive existential quantifier to conditional antecedents.\footnote{Note that, in particular, the rex fragment includes so-called \emph{weak} inquisitive first-order logic, \textsf{InqWQ}, where $\lori$ is the only inquisitive operator \citep[see][]{Conti:26,Ciardelli:25}.} As we should expect, one of the formulas that played a key role in our proofs, namely the formula $\eta$, does not belong to either of these fragments: it is an implication with an inquisitive antecedent (thus lying outside clant), and with an inquisitive existential quantifier in the consequent (thus lying outside rex). One notable thing about $\eta$ is that it involves left-nested implications: in the definition, these are hidden within the dependence formulas $\dep{a}{b}$ and $\dep{b}{a}$, which in our setting are shorthands for the inquisitive implications $\existsi x(x=a)\to\existsi x(x= b)$ and $\existsi x(x=b)\to\existsi x(x= a)$.\footnote{Equivalently, as discussed at the very end of Section \ref{sec:background}, they could be taken as shorthand for the inuisitive implications $\forall x?(x=a)\to\forall x?(x= b)$ and $\forall x?(x=b)\to\forall x?(x= a)$.} A natural question, then, is whether disallowing left-nesting of implications suffices to obtain a tame fragment. Assuming that negation is taken as primitive (rather than defined as $\neg\phi:=(\phi\to\bot)$), this fragment is the counterpart of the so-called NNIL (``No Nestings of Implications on the Left'') fragment of intuitionistic first-order logic, which has been studied since the work of \cite{Visser:85} \citep[see also][]{Benthem:94,Ilin:21}. It is not hard to show that, up to equivalence, the NNIL fragment of \inqbq\ is a proper extension of the clant fragment; thus, a tameness result for this fragment would subsume and extend the known result for clant. Finally, a question in this area is whether it is possible to find a tame fragment of \inqbq\ that includes both the clant fragment and the rex fragment.


A different route to obtaining a tame version of inquisitive first-order logic proceeds, not by restricting the language, but by generalizing the semantics. It is well-known that second-order logic becomes tame relative to so-called Henkin semantics, where second-order quantifiers do not range over the full powerset of the domain, but only over a designated family of subsets closed under certain operations. One might similarly define a Henkin semantics for \inqbq, in which only a designated family of sets of worlds are available as information states in the semantics. As in the case of second-order logic, certain closure conditions on the available sets would have to be imposed to secure the validity of certain natural principles of \inqbq. To the extent that these closure conditions can be stated in (two-sorted) first-order logic, the logic resulting from this generalized semantics---call it $\inqbq_{\textsf H}$---would be compact and recursively enumerable. This logic would, at a minimum, be interesting as a tame approximation of \inqbq\ from below: every validity of $\inqbq_{\textsf H}$ would also be a validity of \inqbq, though the opposite would not hold in general.
\footnote{Thanks to 
Tommaso Moraschini (p.c.) for suggesting this line of work.}

In a different direction, it would be interesting to explore whether our results generalize to non-classical versions of inquisitive first-order logic. For now, the only such logic that has been studied is the intuitionistic version of inquisitive first-order logic \citep{PuncocharNoguera:26}. For this logic, the problems of compactness and recursive enumerability are currently open. It seems plausible that the ideas in this paper can be adapted to settle these problems in the negative, perhaps through an entailment-preserving translations from \inqbq\ to its intuitionistic counterpart.

Finally, an important problem about \inqbq\ posed in \cite{Ciardelli:23book} is left open by our results: is every invalidity of \inqbq\ refutable in a model that has a countable universe and a countable domain? In other words, is the logic of countable models identical to the logic of arbitrary models? It would be interesting to explore whether ideas similar to those we used in our proofs can help with this problem and other related problems concerning potential L\"owenheim-Skolem-like properties of \inqbq.


\bibliographystyle{natbib}
\bibliography{inquisitive}

\end{document}